\documentclass[12pt]{amsart}
\usepackage{amsmath}
\usepackage{amssymb}
\usepackage{amsthm}
\usepackage{amsfonts} 
\usepackage{hyperref}
\usepackage{graphicx}

\numberwithin{equation}{section}

\allowdisplaybreaks

\newtheorem{theorem}{Theorem}[section]
\newtheorem{proposition}[theorem]{Proposition}
\newtheorem{lemma}[theorem]{Lemma}
\newtheorem{remark}[theorem]{Remark}

\renewcommand{\epsilon}{\varepsilon}
\renewcommand{\phi}{\varphi}

\newcommand{\abs}[1]{\left\vert #1\right\vert}
\newcommand{\1}[1]{{\mathbf 1}{\{#1\}}}
\newcommand{\R}{\mathbb{R}}

\newcommand{\Z}{\mathbb{Z}}

\newcommand{\CC}{\mathcal{C}}
\newcommand{\LL}{\mathcal{L}}
\newcommand{\TT}{\mathcal{T}}
\newcommand{\GG}{\mathcal{G}}

\newcommand{\WW}{\mathfrak{W}}

\newcommand{\V}{\mathcal{V}}

\newcommand{\I}{\mathcal{I}}

\newcommand{\IP}{\mathbb{P}}

\newcommand{\bP}{\mathbf{P}}
\newcommand{\bE}{\mathbf{E}}

\newcommand{\Po}{{\mathrm P}^{\omega}}

\newcommand{\Psrw}{P^{\text{\tiny{SRW}}}}

\newcommand{\bv}{\mathbf{v}}
\newcommand{\bk}{\mathbf{k}}
\newcommand{\bj}{\mathbf{j}}

\newcommand{\hg}{{\hat\gamma}}
\newcommand{\hG}{{\widehat\GG}}
\newcommand{\hH}{{\widehat H}}

\newcommand{\hX}{{\widehat X}}

\newcommand{\cyl}{\mathsf{Cyl}}

\newcommand{\capa}{\mathop{\mathrm{cap}}}

\newcommand{\lf}{\lfloor}
\newcommand{\rf}{\rfloor}

\newcommand{\xe}{{x^{\text{entry}}}}
\newcommand{\xt}{{x^{\text{trap}}}}

\newcommand{\eps}{\varepsilon}

\title[]{Biased random walks on the interlacement set}

\author[A.~Fribergh]{Alexander FRIBERGH}
\address{Universit\'e de Montr\'eal,
Department of Mathematics and Statistics,
 Pavillon Andr\'e-Aisenstadt  
    2920, chemin de la Tour, Montr\'eal H3T 1J4,
Canada} 
\email{fribergh@dms.umontreal.ca}

\author[S.~Popov]{Serguei POPOV}
\address{Department of Statistics, Institute of Mathematics,
 Statistics and Scientific Computation, University of Campinas --
UNICAMP, rua S\'ergio Buarque de Holanda 651,
13083--859, Campinas SP, Brazil} 
\email{popov@ime.unicamp.br}


\begin{document}

\begin{abstract}
We study a biased random walk on the interlacement 
set of~$\Z^d$ for~$d\geq 3$. 
Although the walk is always transient, we can show, in the case $d=3$, that
for any value of the bias the walk 
has a zero limiting speed and actually moves 
slower than any power. 
\end{abstract}

\maketitle

\section{Introduction}
\label{s_introres}

The model of 
random interlacements was recently
 introduced by Sznitman in~\cite{Szn10},
and detailed accounts can be found in the survey~\cite{CT12} and
the recent book~\cite{DRS14}.
Loosely speaking, random interlacements
in~$\Z^d$, $d\geq 3$, is a stationary 
Poissonian soup of (transient) doubly infinite simple random
walk trajectories. The \emph{level} of random interlacements
 is an additional 
parameter~$u>0$ entering the intensity measure of the Poisson 
process; as the value of~$u$ increases,
more trajectories are added.
The sites of~$\Z^d$ that are not touched by the trajectories 
constitute
the \emph{vacant set}~$\V^u$ and the union of trajectories
is the \emph{interlacement set}~$\I^u$, so that $\V^u=\Z^d\setminus\I^u$.
It is possible to show that~$\I^u$ is connected 
for all~$u>0$ a.s., cf.\ Theorem~1.5 of~\cite{CP}. In fact, it is possible to construct
the random interlacements simultaneously for all $u>0$
in such a way that $\V^{u_1}\subset \V^{u_2}$ if $u_1>u_2$.
We refer to the above references for the formal definitions (we will give a constructive definition in the next section).

Describing the geometrical properties of a random environment such as the interlacement set 
can be done in several ways. One of the methods which 
has been popular in last decade is to study random walks 
in random environments (RWRE). 
In this paper we aim to do just 
that by considering the biased random walk on the interlacement set. 

This is not the first study of RWRE on the interlacement set. Indeed 
it has been shown in~\cite{PRS15} that the invariance principle 
holds for the simple random walk on the interlacement
cluster; by which we mean that ``typical'' displacement of the 
particle by time~$n$ is of order~$\sqrt{n}$ and that the rescaled
 process converges to a Brownian motion. This result is similar to 
what is observed for the simple random walk on the supercritical 
percolation cluster \cite{BB,MP,SS}.

In the context of the biased random walk on supercritical 
 percolation clusters in~$\Z^d$, it has been shown that 
the walk experiences a phase transition from a positive-speed 
phase (for small biases) to a zero-speed regime for large biases 
in which the walk moves as~$n^{\gamma}$ for some~$\gamma<1$ (see~\cite{FH}). We show that, 
for biased random walks
 on the interlacement set of~$\Z^3$, 
the situation is radically different in the sense that for 
any biases the walk has zero-speed 
 and actually moves slower than any power in~$n$.

\subsection{The model}
\label{s_model}
For $x,y\in \Z^d$, let $x\cdot y$ be the usual scalar product,
and we denote by $e_1,\ldots,e_d$ the unit vectors
of the canonical orthonormal basis. Also, $\|\cdot\|$ stands
for the Euclidean norm. We write $x\sim y$ whenever $\|x-y\|=1$,
i.e., $x$ and~$y$ are neighbours in~$\Z^d$.

Let us denote by~$\Psrw_x$ the law of the $d$-dimensional simple random
walk $(S_n,n\geq 0)$ started from~$x$. 

For any $A\subset\Z^d$ let (using the convention $\min \emptyset = +\infty$) 
\begin{align}
T_A &= \min\{k\geq 1: S_k\in A\} \label{hitting_t}
\end{align}
be the hitting time of~$A$, and write $T_x:=T_{\{x\}}$ 
for $x\in\Z^d$. 
We define the harmonic measure
\[
 e_A(x) = \Psrw_x[T_A=\infty]\mathbf{1}_A;
\]
the \emph{capacity} of~$A$ is defined by
\[
 \capa(A) = \sum_{x\in A} e_A(x),
\]
 see e.g.\ Section~6.5 of~\cite{LL10}.

Let us give a ``constructive''
description of random interlacements at level~$u$ observed
on a finite set~$A$. Namely,
\begin{itemize}
 \item take a Poisson($u\capa(A)$) number of particles;
 \item place these particles on the boundary of~$A$
 independently, with law
 $\overline{e}_A = \big((\capa A)^{-1} e_A(x), x\in A\big)$;
 \item let the particles perform independent simple random walks
 (by transience, each walk only leaves a finite trace
 on~$A$).
\end{itemize}
As a consequence of the above, we obtain the following useful
identity:
\begin{equation}
\label{eq_vacant>3}
 \bP[A\subset \V^u] = \exp\big(-u \capa(A)\big)
\quad \text{for all finite $A\subset\Z^d$}.
\end{equation}
For fixed~$u$, 
we also define $\bP_0[\,\cdot\,]:= \bP[\,\cdot\mid 0\in\I^u]$
to be the law of the interlacement set conditioned to contain 
the origin.

Now, we define the biased random walk on the interlacement
cluster, which is the main object of study of this paper.
Fix a parameter~$\beta>1$ (which accounts for the bias).
Let us define the \emph{conductances}
on the edges of~$\Z^d$ in the following way:
\[
 c(x,y) = \begin{cases}
           \beta^{\max(x\cdot e_1,y\cdot e_1)}, & \text{if }x,y\in\I^u, x\sim y,\\
           0, & \text{otherwise},
          \end{cases}
\]
and we call the collection of all conductances
$\omega = \big\{c(x,y), x,y\in\Z^d\big\}$ the random environment.
Consider a random walk $(X_n, n\geq 0)$ in this environment
of conductances; i.e., its transition probabilities
are given by
\begin{equation}
\label{def_trans_probs}
 q^{\omega}(x,y):=\Po[X_{n+1}=y \mid X_n=x] = 
\begin{cases}
 \frac{c(x,y)}{\sum_{z}c(x,z)}, & \text{if } x,y\in\I^u, x\sim y,\\
0, & \text{otherwise}
\end{cases}
\end{equation}
(the superscript in~$\Po$ indicates that we are dealing
with the ``quenched'' probabilities, i.e., when the underlying 
graph is already fixed). As usual, we abbreviate 
$\Po_x[\,\cdot\,]:=\Po[\,\,\cdot\mid X_0=x]$.
For the sake of cleanness, we work under the measure~$\bP_0$;
then, we are able to choose the
starting point~$X_0$ to be the origin.
Let us also define
$\IP[\cdot] = \int \Po_0[\cdot] \, d\bP_0$ to be the \emph{averaged}
(a.k.a.\ \emph{annealed}) probability measure for the 
walk starting at the origin. 


\subsection{Results}
\label{s_results}
The first result is that the random walk is transient. 
\begin{theorem}
\label{theorem_trans}
For $d\geq 3$, we have 
\[
\lim_{n\to \infty} \|X_n\|=\infty, \qquad 
\text{$\IP$-a.s.},
\]
for any fixed drift~$\beta>1$ and any intensity parameter~$u>0$
of the random interlacements.
\end{theorem}

\begin{remark} We believe that
 $\lim_{n\to \infty} \|X_n \cdot e_1 \|=\infty$ $\IP$-a.s.\ 
for $d\geq 3$. The natural way of proving 
this would be to adapt the proof of Lemma 1.1 in~\cite{Sznitman}. 
However, this proof requires an estimate on the number of 
left-right crossings (i.e., in the direction~$e_1$) of a large box, 
which seems to be difficult to obtain in the case of 
random interlacement (this estimate is used in equation~(1.35)
 of~\cite{Sznitman}). We decided not to pursue this 
in this paper, to be able to focus on the surprising behaviour of 
the speed. 
\end{remark}

Our main result is that, in three dimensions, the biased random walk
on the interlacement cluster has subpolynomial speed:
\begin{theorem}
\label{theorem_3d}
For $d=3$, we have 
\[
\lim_{n\to \infty} \frac{\ln \|X_n\|} {\ln n} =0, \qquad 
\text{$\IP$-a.s.},
\]
for any fixed drift~$\beta>1$ and any intensity parameter~$u>0$
of the random interlacements.
\end{theorem}

As mentioned in the introduction, 
this picture is very different from the one we would get 
by considering the simple random walk on a supercritical 
percolation cluster in $\Z^d$. 
As will become clear in the course of the proof, the above result is 
genuinely three-dimensional. Indeed, using the same methods as 
for our main result, it is possible to show that, in dimensions 
$d\geq 4$,
\emph{for large values of the bias} the walk still has zero speed 
(see Theorem~\ref{t_big_d_0speed} below); but we conjecture that 
if~$\beta$ is close enough to~$1$, then the biased random walk 
should have positive 
speed just as in the case of the biased random walk on percolation 
clusters (see~\cite{FH}). 
However, it is unclear how to prove this result because of the 
difficulties to build a regeneration structure for the walk,
due to the lack of the independence property of the environment.

\begin{theorem}
\label{t_big_d_0speed}
 Let $d\geq 4$, and let $\eps>0$ be arbitrary.
Then, for all large enough $\beta>1$
(depending on~$d$ and~$\eps$) it holds that 
\[
\limsup_{n\to \infty} \frac{\ln \|X_n\|} {\ln n} \leq \eps, \qquad 
\text{$\IP$-a.s.}.
\]
\end{theorem}


Let us emphasize that, apart from the very last section, 
this paper is entirely devoted to the case $d=3$.

\section{Preliminary estimates}
\label{s_prelim}
We start by introducing some further notation.
Given a set~$V$ of vertices of $\Z^d$ we denote by $\abs{V}$ its cardinality.
We define the inner boundary of~$V$ as
\[
\partial V= \big\{x \in V : y\notin  V,~x\sim y\big\}.
\]

For any $x\in \Z^3$ and $L \geq 1$, we define 
the ball centered in~$x$ and with radius~$L$ as
\[
B_x(L)=\bigl\{z\in \Z^3 : \|z-x\| < L \bigr\}.
\]

The positive constants 
(not depending on~$n$)
are denoted by $\gamma,\gamma',\gamma_1,\gamma_2,\gamma_3$ etc.

\subsection{Connectedness and exit probabilities}
\label{s_connectedness}

\subsubsection{Exit probabilities of large cones}
For $L_{1,2,3}\geq 0$ let us denote
\[
\cyl (-L_1,L_2,L_3)=\big \{z\in \Z^3,\ -L_1\leq z\cdot e_1 
\leq L_2,\abs{z\cdot e_i}\leq L_3 \text{ for $i\in\{2,3\}$}\big\},
\]
and 
\[
\partial^+ \cyl (-L_1,L_2,L_3)=\big\{z\in \Z^3,\ z\cdot e_1 = L_2,
\abs{z\cdot e_i}\leq L_3 \text{ for $i\in\{2,3\}$}\big\}.
\]
For any $x\in \Z^3$, we write 
$\cyl _x(-L_1,L_2,L_3)=\{y\in \Z^3,\ y=x+z 
\text{ with } z\in \cyl (-L_1,L_2,L_3)\}$ and similarly 
 for $\partial^+ \cyl _x(-L_1,L_2,L_3)$. 

\begin{lemma}
\label{lem_cerny}
For any $\alpha \in (0,1]$ there exists $\delta>0$ 
such that for any fixed~$u>0$
and all large enough~$n$
\begin{align}
\lefteqn{
\bP_0\big[0\text{ is connected to } 
\partial^+  \cyl (-n^{\alpha},n,
 n^{\alpha})  \text{ in~$\I^u\cap 
\cyl (-n^{\alpha},n, n^{\alpha})$} \big] 
}
\hphantom{*****************************}\nonumber\\
&
\geq 1-\exp(-\gamma n^\delta).
\label{connect_L21}
\end{align}
\end{lemma}
\begin{proof}
Let $\rho_u(x,y)$ be the graph distance (in~$\I^u$) between $x,y\in\I^u$,
and let $z_0=0, z_1=k_1 e_1, z=k_2 e_1,\ldots $ be the sites of~$\I^u$
lying on the ray $\{ke_1: k\in\Z_+\}$
(where, naturally, $0<k_1<k_2<\ldots$). 
It holds that the three-dimensional capacity of 
a ``segment'' $\{k e_1: k\in [0,h]\}$ is of order~$\frac{h}{\ln h}$
(cf.\ e.g.\ Proposition~2.4.5 of~\cite{Law91}),
so, for fixed~$u$ we obtain from~\eqref{eq_vacant>3} that
\begin{equation}
\label{dist_conseq}
 \bP_0[k_{m+1}-k_m> h] \leq \exp\Big(-\gamma'\frac{h}{\ln h}\Big).
\end{equation}
Write, with small enough~$\eps>0$
\begin{align}
\bP_0[\rho_u(z_m, z_{m+1})> s]
 &\leq
\bP_0[k_{m+1}-k_m>\eps s] 
+ \bP_0[\text{there exists }
y\in B_{k_m e_1}(\eps s) \nonumber\\
&\phantom{*****************}
\text{ such that }\rho_u(k_m e_1,y)>s]
\nonumber\\ 
 &\leq \exp\Big(-\gamma'\eps\frac{s}{\ln s}\Big)
  + \gamma'' \exp(-\gamma_1 s^\delta),
 \label{graph_dist} 
\end{align}
where we used~\eqref{dist_conseq} to bound the first term
and Theorem~1.3 of~\cite{CP} to bound the second one.

Then, observe that
the event in~\eqref{connect_L21} contains the event
\[
 \Big\{\rho_u(z_m, z_{m+1})\leq \frac{n^\alpha}{3} 
\text{ for all }m=0,\ldots,n\Big\}.
\]
The claim now follows from~\eqref{graph_dist} and the union bound.
\end{proof}

Define the cone (see Figure~\ref{CM})
\begin{equation}\label{def_cone}
\CC_M(n) =\big\{x\in \Z^3,\ \abs{x\cdot e_1}\leq n, 
\abs{x\cdot e_i}\leq M(n-x\cdot e_1) \text{ for }i\in 
\{2,3\}\big\},
\end{equation}
and its ``positive'' and ``negative'' boundaries
\begin{align*}
\partial^- \CC_M(n) &=\big\{x\in \Z^3,\ x\cdot e_1 =- n, 
\abs{x\cdot e_i}\leq M(n-x\cdot e_1) \text{ for }i\in \{2,3\}\big\}
\\
\partial^+ \CC_M(n) &= \partial \CC_M(n)\setminus\partial^- \CC_M(n).
\end{align*}

\begin{figure}
 \centering \includegraphics{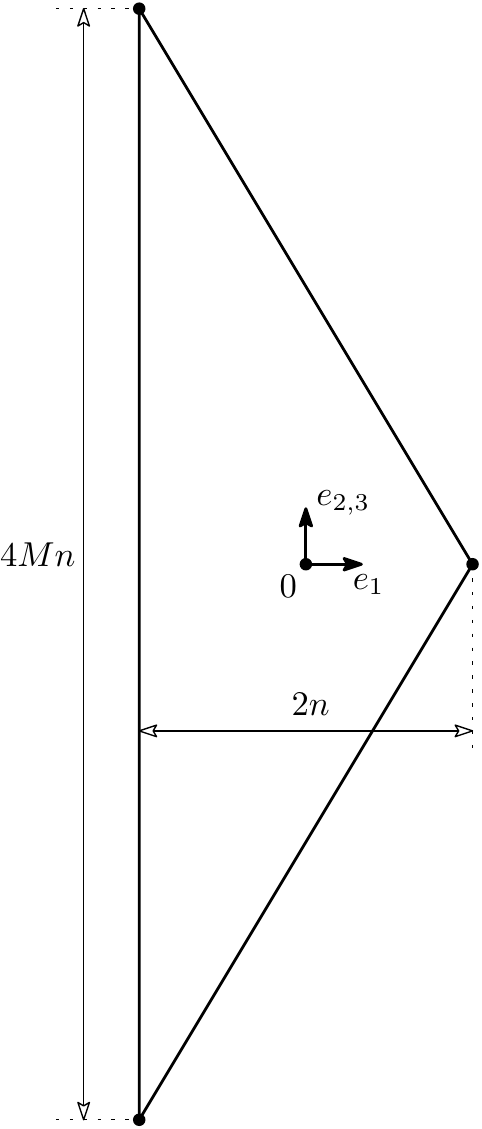} 
\caption{The cone $\CC_M(n)$}
\label{CM}
\end{figure}

First, we prove that it is very probable that the walker exits large
 cones on the positive side:
\begin{lemma}
\label{lem_backtrack}
For any $M<\infty$, we have 
\[
\IP[T_{\partial^- \CC_M(n)}<T_{\partial^+ \CC_M(n)}] 
\leq \gamma \exp(-\gamma n^{\delta}),
\] 
for some $\delta<1$.
\end{lemma}
\begin{proof}
Let us consider the event
\[
G_n :=\big\{0\text{ is connected in~$\I^u\cap 
\cyl (-n^{\alpha},n, n^{\alpha})$ to } 
\partial^+  \cyl (-n^{\alpha},n,
 n^{\alpha}) \big\}.
\]

For any environment $\omega$, we may 
use equation (4) in~\cite{BGP} (or also exercise~2.36 in~\cite{LP}) 
 to see that
\begin{equation}\label{bt_1}
P_0^{\omega}[T_{\partial^- \CC_M(n) }< T_{\partial^+ \CC_M(n) }] 
\leq \frac {C^{\omega}(0\leftrightarrow \partial^- \CC_M(n) )}{C^{\omega}(0\leftrightarrow \partial \CC_M(n))},
\end{equation}
where~$C^\omega$ stands for the effective conductance
in the (weighted) interlacement graph restricted on~$\CC_M(n)$.
In an environment $\omega\in G_n$, we know that there is 
a simple path from $0$ to $\partial^+  \cyl (-n^{\alpha},n,
 n^{\alpha})$ within $ \cyl (-n^{\alpha},n,
 n^{\alpha})$ and because of our definition of $\CC_M(n)$ 
this path has to cross $\partial \CC_M(n)$ before reaching 
$\partial^+  \cyl (-n^{\alpha},n,  n^{\alpha})$. 
This implies that for $\omega\in G_n$ there exists 
a path~$\mathcal{P}$, composed of $v_0=0,v_1,\ldots, v_{i_0}$ 
such that
\begin{enumerate}
\item $v_{i_0}\in \partial \CC_M(n)$,
\item $i_0\leq \abs{\cyl (-n^{\alpha},n,n^{\alpha}) }\leq \gamma n^3$,
\item $v_j\cdot e_1\geq -n^{\alpha}$ for any $j\leq i_0$.
\end{enumerate}

We recall that Rayleigh's monotonicity principle (see
Section~2.4 of~\cite{LP}) 
implies that closing edges in a graph decreases all effective 
conductances. Hence, we know that the effective conductance  
$C^{\omega}(0\leftrightarrow \partial \CC_M(n))$ can be
 lower bounded by the conductance of the path~$\mathcal{P}$ 
which is at least $cn^{-3}\beta^{-n^{\alpha}}$ 
(where~$\alpha<1$) since it is composed of at most~$\gamma n^3$ edges 
in series which have conductance at least $c\beta^{-n^{\alpha}}$. 
Hence, in an environment~$\omega$ belonging to the event 
appearing in Lemma~\ref{lem_cerny}, we have
\begin{equation}
\label{bt_2}
C^{\omega}(0\leftrightarrow \partial \CC_M(n))
  \geq cn^{-3}\beta^{-n^{\alpha}}.
\end{equation}
Rayleigh's monotonicity principle 
 also implies that merging vertices together increases 
effective conductances. Let us merge all vertices of 
$ \CC_M(n)\setminus \partial^- \CC_M(n)$ 
(which contains the origin) into~$\Delta_1$ 
and all vertices of $\partial^- \CC_M(n)$ into~$\Delta_2$;
we can use Rayleigh's monotonicity principle to see that 
$C^{\omega}(0\leftrightarrow \partial^- \CC_M(n) )
\leq C^{\omega}(\Delta_1\leftrightarrow \Delta_2)$. 
The latter can be upper bounded by seeing that~$\Delta_1$ 
and~$\Delta_2$ are linked by at most~$\gamma n^3$ edges of 
conductances at most~$\gamma \beta^{-n}$. Hence
\begin{equation}
\label{bt_3}
C^{\omega}(0\leftrightarrow \partial^- \CC_M(n) )
\leq \gamma  n^3 \beta^{-n}.
\end{equation}

Putting together~(\ref{bt_1}),~(\ref{bt_2}) and~(\ref{bt_3}) 
we see that, for any~$\omega$ belonging to the event appearing 
in Lemma~\ref{lem_cerny}, we have
\[
\Po_0[T_{\partial^- \CC_M(n) }
< T_{\partial^+ \CC_M(n) }] 
\leq \gamma n^{6}\beta^{-n+n^{\alpha}},
\]
for some $\alpha<1$. Hence, by using Lemma~\ref{lem_cerny}, 
we see that for some~$\delta<1$ we have
\begin{align*}
 \IP[T_{\partial^- \CC_M(n) }< T_{\partial^+ \CC_M(n) }]
&=\IP[G_n ^c] 
+\textbf{E}\big[G_n , \Po_0[T_{\partial^- \CC_M(n) }
< T_{\partial^+ \CC_M(n) }]\big] \\
 &\leq  \gamma' \exp(-\gamma n^{\delta}),
\end{align*}
which implies the result.
\end{proof}

Let us introduce
\begin{equation}
\label{def_Phi_n}
\Phi_n =\bigcap_{i=1}^{n^{2/3}} 
\big\{T_{\partial^+ \CC_M(in^{1/3})}
< T_{\partial^- \CC_M(in^{1/3})}\big\};
\end{equation}
we can then use the union bound and 
Lemma~\ref{lem_backtrack} to see that
\begin{equation}
\label{est_Phi_n}
\IP[\Phi_n ^c] \leq \gamma' \sum_{i=1}^{n^{2/3}}
\exp\big(-\gamma (in^{1/3})^{\delta}\big).
\end{equation}

\subsubsection{Proof of transience}

Let us define
\[
K(n)=\big\{x\in \Z^3,\ x\cdot e_1 >-n,\ |x\cdot e_i| 
\leq n +x\cdot e_1 \text{ for } i\in \{2,3\} \big\}.
\]

\begin{lemma}\label{last_lem}
There exists $\delta>0$ such that for any fixed~$u>0$
and for all~$n$ large enough
\begin{equation}
\label{cyl_connected}
\bP_0\big[0\text{ is connected to infinity in }
\I^u\cap K(n)\big] \geq 1-\exp(-\gamma n^{\delta}).
\end{equation}
\end{lemma}
\begin{proof}
This can be proved quite similarly to Lemma~\ref{lem_cerny}.
Using the same notation,
observe that
the event in~\eqref{cyl_connected} contains the event
\[
 \Big\{\rho_u(z_m, z_{m+1})\leq \frac{n+m}{3} 
\text{ for all }m\geq 0\Big\}.
\]
Again, the claim follows from~\eqref{graph_dist} 
and the union bound.
\end{proof}

Let us now prove Theorem~\ref{theorem_trans}.
\begin{proof}
Using Lemma~\ref{last_lem} and Borel-Cantelli we can show that
$\bP_0$-a.s.\ there exists~$N(\omega)$ such that~$0$ 
belongs to an infinite simple path lying in $\I^u\cap K(N(\omega))$.

For any $k \geq 0$, this path will contain at most~$\gamma k^2$ 
edges of conductances $\gamma \beta ^{k-N(\omega)+1}$. 
This means, using Rayleigh's monotonicity principle and the 
law of resistances in series 
(see Chapter~2 of~\cite{LP}), we can prove that 
\[
R^{\omega}(0\leftrightarrow \infty) \leq \gamma \sum_{k\geq 0}  k^2 \beta^{N(\omega)-k+1}  <\infty.
\]

This means that the resistance from~$0$ to~$\infty$ is finite, 
which means that the random walk is transient (see Chapter~2 of~\cite{LP}).
\end{proof}

\subsection{Traps}
\label{s_traps}
Let us remind the reader that we are working in three dimensions.

As usual, the method for proving that the biased random 
walk has zero speed is showing that it will encounter
a \emph{trap}, i.e., a part of the environment where 
the random walk will stay for a long time. 
For a biased random walk, this consists in looking for
 dead-ends in the direction~$e_1$ from, 
once the biased walk has entered such a dead-end, 
it will have to fight against the drift to exit the trap.

 From now on we assume that~$M$ is not too small, say, $M\geq 10$.
Let us introduce the ``quiver'' set 
\begin{equation}
\label{def_PT}
Q(x,M,n)= \partial\cyl_x\big(0, \lf M\ln n\rf+1, (\ln n)^{3/4}\big),
\end{equation} 
where $\lf\cdot\rf$ stands for the integer part. 
Also, for $x\in\Z^3$ we denote
 $\xe= \xe(x,M,n)=
x+ \lf\frac{3}{M} \ln n\rf e_1$ and 
$\xt =\xt (x,M,n)
=\xe+\lf M\ln n\rf e_1$. 
Observe that $\xe \in Q(\xe , M, n)$
and~$\xt$ is strictly inside $Q(\xe,M,n)$.

Our goal is to find a trapping structure for the walk. Let us denote 
$\TT (x,M,n)$ the event that there exists a 
trap at~$x$
(see Figure~\ref{f_trap}),
defined in the following way: let
\begin{align*}
 \TT^{(1)}(x,M,n) &= \big\{\text{$\mathcal{I}_u$ intersects
              $Q(\xe,M,n)$ only at~$\xe $}\big\},\\
 \TT^{(2)}(x,M,n) &= \big\{x+je_1\in\mathcal{I}_u, \text{ for all }
  0\leq j\leq \lf\textstyle\frac{3}{M} \ln n\rf\big\},\\
  \TT^{(3)}(x,M,n) &= \big\{\xt\in \mathcal{I}_u\big\},
\end{align*}
and we define 
$\TT (x,M,n)= \TT^{(1)}(x,M,n)\cap\TT^{(2)}(x,M,n)\cap\TT^{(3)}(x,M,n)$.

Not focusing, for now, on technicalities, the key part of the event 
$\TT (x,M,n)$ is that, not far away from~$x$ in the direction
 of the drift,
there is a structure in~$\mathcal{I}_u$ creating a dead end for the
biased random walk. 

\begin{figure}
 \centering \includegraphics[width=\textwidth]{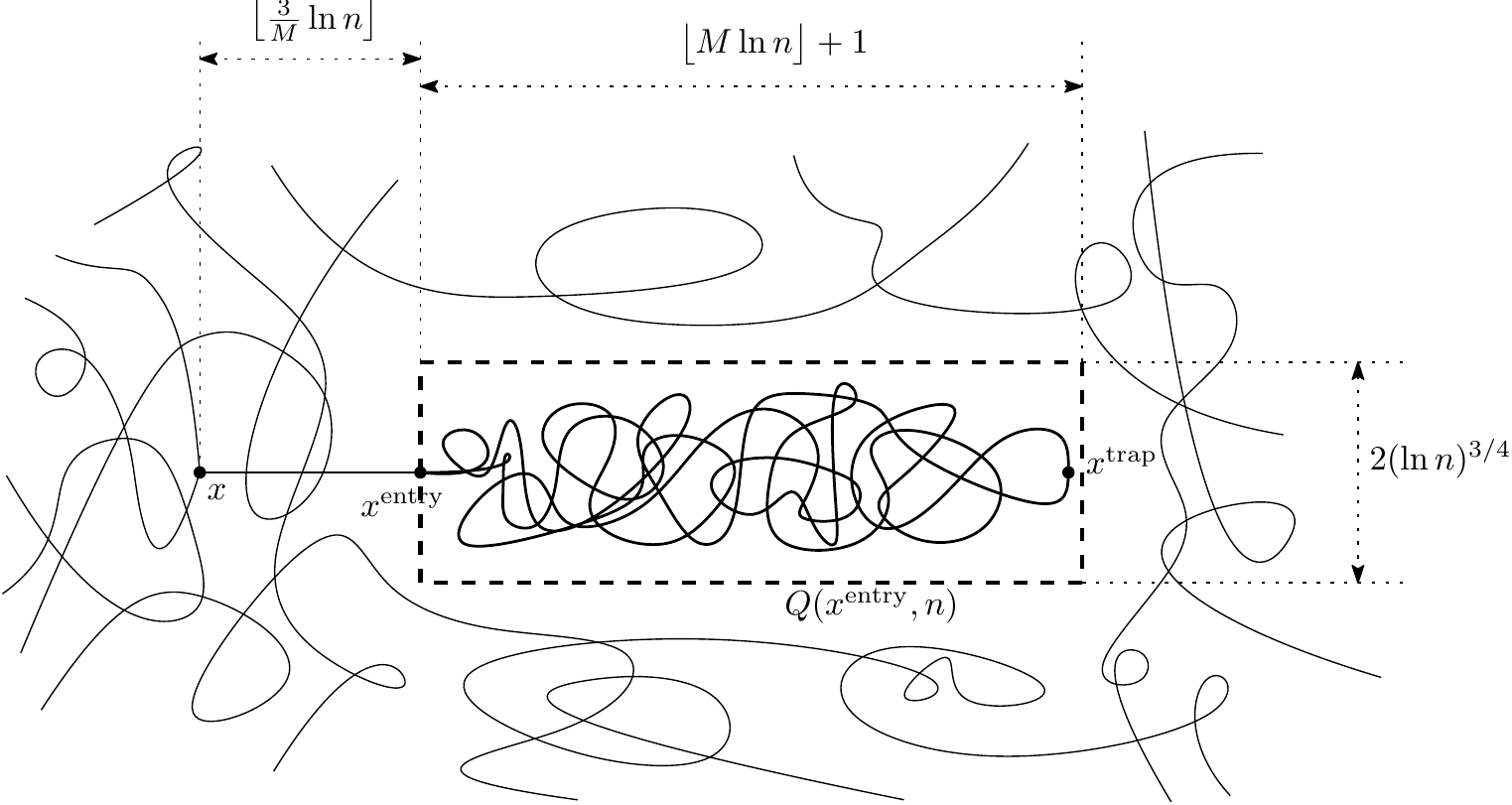} 
\caption{A trap for the random walk~$X$. The solid lines/curves
are the interlacements; the trajectory of the RWRE~$X$ is not shown
on the picture.}
\label{f_trap}
\end{figure}

Such a structure can appear if, for example
\begin{enumerate}
\item all the walk traces forming $\mathcal{I}_u$ except one avoid 
$Q(\xe,M,n)$;
\item simultaneously, the one remaining walk's trace 
has a behavior such that the two last conditions present
 in the definition of $\TT(x,M,n)$ are satisfied.
\end{enumerate}

These two types of events are the ones 
we are going to study in order to understand the likelihood 
of finding traps.

%
%

Recall that $\Psrw_x$ stands for the law of simple random
walk~$(S_n,n\geq 0)$ started from~$x$. 
Let $\sigma_1=\min\{k> T_{\xe }: S_k=\xt\}$
be the time of the first visit to~$\xt$ after
visiting~$\xe $,
$\sigma_2=\min\{k> T_{\xe }: S_k=\xe \}$
be the moment when~$\xe $ is visited for the second time,
and $\sigma_3=\min\{k> T_{\xe }: S_k\in Q(\xe ,M,n) \}$
be the moment when~$Q(\xe ,M,n)$ is first visited \emph{after}~$T_{\xe }$.
Define the event
\begin{align}
 E_{x,M,n} &= \big\{S_i=x+ie_1, i=1,\ldots,\lf\textstyle\frac 3M \ln n\rf,
 \sigma_1<\sigma_2
 = \sigma_3\ ,\nonumber\\
 & \qquad\qquad\qquad\qquad\qquad S_{\sigma_2+j}=\xe -je_1,
   j=1,\ldots,\lf\textstyle\frac 3M \ln n\rf\big\};
\label{df_loop}
\end{align}
that is, the trajectory makes a loop that first goes straight
from~$x$ to~$\xe $, then goes to~$\xt$
and returns to~$\xe $ strictly
inside~$Q(\xe ,M,n)$, and then returns straight to~$x$,
see Figure~\ref{f_trap}.

\begin{lemma}
\label{l_dig_a_trap}
There is a constant~$\gamma>0$ such that 
\begin{equation}
\label{first_walk}
\Psrw_{x}[E_{x,M,n}]\geq n^{-\gamma /M}
\end{equation}
for all large enough~$n$.
\end{lemma}

\begin{proof}
 Clearly, it holds that 
\begin{equation}
\label{go_straight}
 \Psrw_{x}\big[S_j=x+je_1, 
S_{\sigma_2+j}=\xe -je_1,
   j=1,\ldots,\lf\textstyle\frac 3M \ln n\rf\big]
= 6^{-2\lf \frac 3M \ln n\rf},
\end{equation}
so it remains to find a lower bound on the probability
that the trajectory behaves as it should inside~$Q(\xe ,M,n)$.

Define 
\begin{align*}
 \sigma'_1 &= \min\big\{k\geq 1: S_k\cdot e_1 = \xe\cdot e_1
    + \lf M\ln n\rf - \lf \textstyle\frac{1}{M}\ln n\rf\big\},\\
 \sigma'_2 &= \min\big\{k\geq 1: S_k\cdot e_1 = \xe\cdot e_1
    + \lf \textstyle\frac{1}{M}\ln n\rf\big\}.
\end{align*}
Then, by the (strong) Markov property it holds that
(recall that $\xe\in Q(\xe ,M,n)$)
\begin{align}
\lefteqn{ \Psrw_\xe\big[T_\xt<T_\xe = T_{Q(\xe ,M,n)}\big]}
\nonumber\\
 &\geq
\Psrw_\xe\big[\sigma'_1< T_{Q(\xe ,M,n)},(S_1-\xe)\cdot e_1=1\big]
\nonumber\\
 &\quad {}\times\Psrw_\xe\big[T_\xt-\sigma'_1
\leq 3\lf \textstyle\frac{1}{M}\ln n\rf < T_{Q(\xe ,M,n)}
 \mid (S_1-\xe)\cdot e_1=1, \nonumber\\
 &\phantom{**********************************}
\sigma'_1< T_{Q(\xe ,M,n)} \big]
\nonumber\\
& \quad {}\times \Psrw_\xt\big[\sigma'_2< T_{Q(\xe ,M,n)}\big]
\nonumber\\
& \quad {}\times
 \Psrw_\xt\big[T_\xe-\sigma'_2\leq 3\lf \textstyle\frac{1}{M}\ln n\rf
  < T_{Q(\xe ,M,n)}
 \mid \sigma'_2< T_{Q(\xe ,M,n)}\big]
\nonumber\\
&=: F_1 \times F_2 \times F_3 \times F_4.
\label{strongMarkov_decomp}
\end{align}
\begin{equation}
\label{F24}
 F_2 \wedge F_4 \geq 6^{-3\lf \textstyle\frac{1}{M}\ln n\rf}
\end{equation}
(just follow a fixed path of length at most 
$3\lf \textstyle\frac{1}{M}\ln n\rf$ that leads to~$\xe$ or~$\xt$).

In order to estimate the other two terms in~\eqref{strongMarkov_decomp},
denote $\tilde{S}_k^{(i)}=S_k\cdot e_i$, $i=1,2,3$ and
\[
S_k^{(i)}=\tilde{S}_{\theta_k^{(i)}}^{(i)} \text{ with } \theta_k^{(i)}=\inf\{j >\theta_{k-1}^{(i)}, \ \tilde{S}_j^{(i)}-\tilde{S}_{j-1}^{(i)} \neq 0\},
\]
initialized with $\theta_0^{(i)}=0$. In words,
$S_k^{(i)}$ records the successive steps of the SRW~$S$ in the $i$-th 
direction. Obviously, we have that $S^{(1)}$, $S^{(2)}$ and~$S^{(3)}$ 
are independent (the same is not true for $\tilde{S}^{(i)}$).

 Let $T^{(i)}$'s be the corresponding hitting times defined analogously to~\eqref{hitting_t}. We then write
\begin{align*}
\lefteqn{ \Psrw_0\big[T^{(1)}_{\lf M\ln n\rf - 
 \lf \frac{1}{M}\ln n\rf} <M^2 \ln^2 n,
    T^{(1)}_{\lf M\ln n\rf - 
 \lf \frac{1}{M}\ln n\rf} < T^{(1)}_0\big] }\\
& = \Psrw_0\big[T^{(1)}_{\lf M\ln n\rf - 
 \lf \frac{1}{M}\ln n\rf} <M^2 \ln^2 n
  \mid  T^{(1)}_{\lf M\ln n\rf - 
 \lf \frac{1}{M}\ln n\rf} < T^{(1)}_0\big] \\
&\qquad \times \Psrw_0\big[T^{(1)}_{\lf M\ln n\rf - 
 \lf \frac{1}{M}\ln n\rf} < T^{(1)}_0\big] .
\end{align*}
Clearly, the second term in the right-hand side of the above
display is bounded below by $1/\lf M\ln n\rf$.
As for the first term, observe that formula~(6)
of~\cite{Stern} implies that the conditional expectation
of the time the simple random walk starting at~$1$ hits~$a\geq 2$, given
that it hits~$a$ before hitting~$0$, is equal to $(a^2-1)/3$.
So, using Chebyshev's inequality for the probability
of the complementary event, we see that the first term
is bounded below by~$2/3$. 
Therefore, we obtain that
\begin{equation}
\label{one-dim_escape}
 \Psrw_0\big[T^{(1)}_{\lf M\ln n\rf - 
 \lf \frac{1}{M}\ln n\rf} <M^2 \ln^2 n,
    T^{(1)}_{\lf M\ln n\rf - 
 \lf \frac{1}{M}\ln n\rf} < T^{(1)}_0\big]
 \geq \frac{2}{3M\ln n},
\end{equation}
which loosely speaking means that the first component has 
probability at least $ \frac{2}{3M\ln n}$ to reach the 
right-hand side of the quiver in time a time less 
than $M^2 \ln^2 n$. Also, denote
\[
 \tau^{(i)}=\min\big\{k: |S_k^{(i)}|=\lf (\ln n)^{3/4} \rf\big\}
\]
 for $i=2,3$.
It holds 
that (see e.g.\ \S21 of Chapter~V of~\cite{Spitzer})
\begin{equation}
\label{conf_1dim}
 \Psrw\Big[ \tau^{(i)} \geq \frac 14 M^2\ln^2 n\Big] 
  \geq \exp\big(-\gamma''M^2(\ln n)^{1/2}\big),
\end{equation}
this means that the second and third coordinates have probability 
at least $\exp\big(-\gamma''M^2(\ln n)^{1/2}\big)$ to stay confined
 in the quiver for a time at least $\frac 14 M^2 \ln^2 n$.

The last remaining step is to notice that by the law of large
 numbers the time-changes $\theta^{(i)}$ are such that
\[
\Psrw\Big[\theta^{(1)}_{M^2 \ln^2 n}\leq \frac 14 M^2 \ln^2 n,\ 
 \theta^{(2)}_{\frac 14 M^2\ln^2 n}>  \frac 18 M^2 \ln^2 n,\ 
 \theta^{(3)}_{\frac 14 M^2\ln^2 n}>  \frac 18 M^2 \ln^2 n\Big] \geq \frac12,
\]
since asymptotically $1/3$ of the steps should be taken in any directions.

Since the time changes are independent of the walks $S^{(j)}$ we can use~\eqref{one-dim_escape} and~\eqref{conf_1dim} to obtain that 
\begin{equation}
\label{F13}
 F_1\wedge F_3 \geq \frac 12 \times \frac{2}{3M\ln n}
 \times \Big(\frac{\gamma'\exp\big(-\gamma''(\ln n)^{1/2}\big)}
 {M\ln n}\Big)^2,
\end{equation}
and this concludes the proof of Lemma~\ref{l_dig_a_trap}.
\end{proof}

\subsection{Finding traps in the interlacement set}
\label{s_res_RI}
First of all, we need to estimate the ``cost'' of having
a trap in some fixed place. 
\begin{lemma}
\label{l_capacity}
In three dimensions, we have
$\capa(Q(x,M,n))\leq \frac{\gamma M\ln n}{\ln\ln n}$.
\end{lemma}
\begin{proof}
 Abbreviate, for now, $Q:=Q(x,M,n)$ and $m:=\ln n$. Also, let~$g(x,y)$ 
be the Green's function of the simple random walk;
it is well known that for all $x,y\in \Z^3$
\begin{equation}
\label{est_Green} 
 g(x,y) = g(y,x) = g(0,y-x) \leq \frac{\hg}{1+\|x-y\|}
\end{equation}
for some positive constant~$\hg$. Let us define the set of functions
\[
 \Sigma^\downarrow = \Big\{\phi\in \R^{\Z^3} : 
    \sum_{y\in Q}g(x,y)\phi(y)\geq 1 \text{ for all }x\in Q\Big\}.
\]
Then it holds that 
\begin{equation}
\label{variational_capa}
  \capa(Q) = \inf_{\phi\in\Sigma^\downarrow}\sum_{x\in Q} \phi(x),
\end{equation}
see Lemma~1.14 of~\cite{DRS14}.

Now, it is elementary to observe that for any~$x\in Q$
\[
 |\{y\in Q: \|x-y\|\in [k,k+1)\}| \geq 
    \begin{cases}
     \gamma_1 k, & \text{ for } k=1,\ldots, \lf m^{3/4}\rf\\
     \gamma_2 m^{3/4}, & \text{ for } k=\lf m^{3/4}\rf+1,
          \ldots, \textstyle\frac{1}{2}Mm.
    \end{cases}
\]
   From this, for any~$x\in Q$ we obtain
\begin{align*}
 \sum_{y\in Q} \frac{1}{\|x-y\|} &\geq \sum_{k=1}^{\frac{1}{2}Mm}
\frac{|\{y: \|x-y\|\in [k,k+1)\}|}{k}\\
 &\geq \gamma_1 m^{3/4} + \gamma_2 \sum_{k=\lf m^{3/4}\rf+1}^{\frac{1}{2}Mm}
 \frac{m^{3/4}}{k}\\
&\geq \gamma_3 m^{3/4} \ln m,
\end{align*}
so (recall~\eqref{est_Green}) a function~$\phi$ 
that equals $\frac{\gamma_4}{m^{3/4}\ln m}$ 
on~$Q$ for large enough~$\gamma_4$, belongs to~$\Sigma^\downarrow$
for all~$n$ large enough. The claim of
 Lemma~\ref{l_capacity} now follows from~\eqref{variational_capa}
 since~$|Q|$ is of order $Mm^{7/4}$.
\end{proof}

One technical difficulty in proving
that we regularly encounter traps
 is that we want to take into account 
the information obtained from the past trajectory of the walk; 
indeed, correlations in the interlacement set have infinite range.
That is, we need to be able to work with the \emph{conditional}
law of the interlacement set, given that inside some finite
set the interlacement configuration is (partially or even completely)
revealed. Next, we formulate a result from~\cite{AP} 
about the conditional decoupling for random interlacements.
We also observe that the unconditional decoupling from~\cite{SLT}
 is not enough in this situation.

With some abuse
of notation, we denote by~$\I^u_A$ 
 the interlacement configuration on level~$u$
restricted on~$A$, i.e., for $x\in A$ we write
 $\I^u_A(x)=1$ whenever $x\in\I^u$.
   
   
\begin{proposition}
\label{p_main3}
Let $u'>u>0$, and let $A_1=B_0(r)$, $A_2\subset\Z^3\setminus B_0(r+s)$;
assume that $\gamma_1s\leq r \leq \gamma_2 s$ 
for some fixed~$\gamma_{1,2}>0$. 
Then, there are positive constants~$\gamma,\gamma'$ depending only on 
dimension, and a (measurable)
set~$\mathcal{G}_{u'}\in\{0,1\}^{A_2}$ such that
\begin{equation*}
\bP \big[\I^u_{A_2}\in\mathcal{G}_{u'}\big]
\geq 1-\exp\big(-\gamma' u' s^{d-2}\big),
\end{equation*}
and for any increasing event~$E$ on the 
interlacements set intersected with~$A_1$, 
 we have
\begin{equation}
\label{e_conditionaldecoupling3}
\bP[E(\I^{u}_{A_1}) \mid \I^{u}_{A_2} ]
\1{\I^u_{A_2}\in\mathcal{G}_{u'}} 
\leq 
 \big(\bP[E( \I^{u + u'}_{A_1})]+\gamma
 \exp(-\gamma' u' s^{d-2})\big)\1{\I^u_{A_2}\in\mathcal{G}_{u'}}. 
\end{equation}
\end{proposition}
\begin{proof}
This is an immediate corollary of Theorem~2.2 of~\cite{AP}.
\end{proof}

The above decoupling result implies the following:
\begin{lemma}
\label{l_find_quiver}
For any~$M$ there exists a constant $g_M>0$ and 
a set $\GG(x,M,n)$ of good environments
on $\CC_M(n)$ such that 
$\bP_0[\GG(x,M,n)]\geq 1-n^{-12}$ and, for $\I^u_{\CC_M(n)}\in \GG(x,M,n)$
\begin{equation}
\label{eq_find_quiver}
 \bP_0\big[Q(\xe,M,n)\subset \V^u\mid \I^u_{\CC_M(n)}\big]
   \geq \exp\Big(-\frac{g_M\ln n}{\ln\ln n}\Big),
\end{equation}
for any $x\in\partial^+\CC_M(n)$
\end{lemma}
\begin{proof}
 Indeed, let us first note that 
one can insert $Q(\xe,M,n)$ inside a ball of radius $M\ln n$
in such a way that the distance between this ball and~$\CC_M(n)$
would be at least $\frac{2}{M}\ln n$. 

Then, use Proposition~\ref{p_main3} with the increasing
event $\{Q(\xe,M,n)\cap \I^u \neq \emptyset\}$, $r=M \ln n$ 
and $s=\frac 2M \ln n$.

Observe that one can choose a large enough~$u'$ in such a way that
the probability of the event $\big(\GG(x,M,n)\big)^\complement$
would be bounded above by any negative power of~$n$. 

An application of~\eqref{eq_vacant>3} together with Lemma~\ref{l_capacity} yields that  for $\I^u_{\CC_M(n)}\in \GG(x,M,n)$
\[
 \bP_0\big[Q(\xe,M,n)\cap \I^u \neq \emptyset \mid \I^u_{\CC_M(n)}\big]
   \leq 1-\exp\Big(-\frac{g_M\ln n}{\ln\ln n}\Big) + Cn^{-12},
   \]
which finishes the proof. 
\end{proof}
   
The second important part for constructing a trap is to find
 an interlacement that actually creates the trap 
inside the cylinder avoided by all the other walks. 
This is our aim for now. We need a result about
``adding a loop to an existing configuration'', which
we now describe. Let~$A$ be a finite subset
of~$\Z^d$, $d\geq 3$ (we formulate this result for general~$d$
since it may be of independent interest). 
Fix any~$x_0\in A$ and let 
$x_0=y_0\sim y_1\sim\cdots \sim y_m=x_0$ be a nearest-neighbour
path that begins and ends in~$x_0$ and such that 
$y_k\in A\setminus\partial A$ for $k=0,\ldots,m$.
In the result below we need a finer control of the 
random interlacements: let $\LL^u(x)$ be the \emph{local time}
at site~$x$ at level~$u$, that is, the sum of local (occupation)
times in~$x$ of all trajectories at level~$u$.
We denote by $\eta\in\{0,1,2,\ldots\}^A$ a generic
configuration of local times
 on~$A$, and by~$\ell$ the configuration
``generated'' by the above loop, i.e., 
$\ell(x)=\sum_{k=1}^m \1{y_k=x}$
Write $(\eta+\ell) (x):= \eta(x)+\ell(x)$, 
and denote by $\LL^u_A(x)$ the local time configuration
on $A\subset \Z^d$.
\begin{lemma}
\label{l_loop}
 For any~$\eta$ such that $\eta(x_0)=1$ we have
\begin{equation}
\label{eq_loop}
 \bP[\LL^u_A = \eta+\ell] \geq (2d)^{-m}\bP[\LL^u_A = \eta]
\end{equation}
(recall that $x_0$ is the initial vertex of the loop, 
and~$m$ is the number of steps in the loop).
\end{lemma}
\begin{proof}
 A \emph{trace} on~$A$ is a finite sequence
$v=(v_0,\ldots,v_s)$ of vertices of~$A$ such that
either $v_{i-1}\sim v_i$ or $v_{i-1},v_i\in\partial A$
for all $i=1,\ldots, s$, and also $v_0,v_s\in\partial A$. 
 For a trace~$v$, define its weight as
\[
 p_v = \prod_{j=1}^s \Psrw_{v_{j-1}}[T_A<\infty,S_{T_A}=v_j]
    \times \Psrw_{v_s}[T_A=\infty];
\]
observe that if $v_{j-1}\in A\setminus\partial A$ then
$\Psrw_{v_{j-1}}[T_A<\infty,S_{T_A}=v_j]$ simply equals $(2d)^{-1}$,
and $\Psrw_{v_{j-1}}[T_A<\infty,S_{T_A}=v_j]\geq (2d)^{-1}$
in case $v_{i-1}\sim v_i$ (the inequality is strict if
both $v_{i-1}$ and~$v_i$ are on the boundary).
For a finite sequence of traces $\bv=(v^{(1)},\ldots,v^{(|\bv|)})$
we define its \emph{total weight} by
\[
 \mathsf{P}[\bv] = e^{-u\capa(A)}\frac{(u\capa(A))^{|\bv|}}{|\bv|!}
  \prod_{k=1}^{|\bv|} \frac{e_A(v_0^{(k)})}{\capa(A)}p_{v^{(k)}}
\]
(in fact, one may cancel $(\capa(A))^{|\bv|}$ in the above formula,
but we prefer to write it this way to make it clearer that the above
equals the probability that $v^{(k)}$'s are the traces left on~$A$
by the trajectories of random interlacements 
ordered with respect to the $u$-coordinate). Write also 
\[
L_{\bv}(x) = \sum_{k,j}\1{v_j^{(k)}=x},
\]
so that $L_{\bv}$ is the (total) local time
of the traces of~$\bv$, and observe that
\begin{equation}
\label{sum_bv}
 \bP[\LL^u_A = \eta] = \sum_{\bv: L_\bv=\eta} \mathsf{P}[\bv].
\end{equation}
Next, for~$\bv$ such that $L_\bv(x_0)\geq 1$ let us define
\begin{align*}
 \bk(\bv) &= \min\{k: \text{ there exists }j
\text{ such that }v_j^k=x_0\}\\
\intertext{and}
\bj(\bv) &= \min\{j: v_j^{(\bk(\bv))}=x_0\}.
\end{align*}
Also, for such~$\bv$ define a sequence of traces
\[
{\hat \bv} = (v^{(1)},\ldots,v^{(\bk(\bv)-1)}, {\hat v}^{(\bk(\bv))},
v^{(\bk(\bv)+1)},\ldots, v^{(|\bv|)}),
\]
where
\[
 {\hat v}^{(\bk(\bv))}_j = 
     \begin{cases}
      v^{(\bk(\bv))}_j, & \text{ if }j\leq \bj(\bv),\\
      y_i, &    \text{ if }j= \bj(\bv)+i,\quad 0<i\leq m,\\
      v^{(\bk(\bv))}_{j-m}, & \text{ if }j> \bj(\bv)+m;
     \end{cases}
\]
in words, $\hat\bv$ is obtained from~$\bv$ by inserting the loop
at the first possible location.
Now, by construction, it holds that 
$p_{{\hat v}^{(\bk(\bv))}} = (2d)^{-m}p_{v^{(\bk(\bv))}}$
(since the whole path lies in $A\setminus\partial A$),
and so $\mathsf{P}[{\hat \bv}]=(2d)^{-m}\mathsf{P}[\bv]$.
Now, $L_\bv=\eta$ implies $L_{\hat\bv}=\eta+\ell$
so, using~\eqref{sum_bv} and the fact that $\bv\mapsto{\hat\bv}$
is an injection
\[
\bP[\LL^u_A = \eta+\ell] 
\geq \sum_{\bv: L_\bv=\eta}\mathsf{P}[{\hat \bv}]
=(2d)^{-m}\bP[\LL^u_A = \eta],
\]
which concludes the proof of Lemma~\ref{l_loop}.
\end{proof}

%

Let~$A=\CC_M(n+3M \ln n)$ so that $Q(\xe,M,n)\subset A$
for all $x\in \partial^+\CC_M(n)$, and, for 
a fixed $\xi_0\in\{0,1\}^{\CC_M(n)}$ define
\begin{align*}
 \WW_1^{\xi_0} &= \big\{\eta\in \{0,1,2,\ldots\}^A:  \LL^u(x)=0
 \text{ for all } x\in Q(\xe,M,n), \\
& \qquad \qquad \qquad
\1{\LL^u(y)\geq 1}=\1{\xi_0(y)=1} \text{ for all }
y\in \I^u_{\CC_M(n)}\big\},\\
\intertext{and}
 \WW_2^{\xi_0} &= \big\{\eta\in \{0,1,2,\ldots\}^A: \TT (x,M,n)
 \text{ occurs on } \eta, \\
& \qquad \qquad \qquad
\1{\LL^u(y)\geq 1}=\1{\xi_0(y)=1} \text{ for all }
y\in \I^u_{\CC_M(n)}\big\}.
\end{align*}
Next, let~$\Phi$ be the set of local times~$\ell$ of loops
belonging to the event~$E_{x,M,n}$, recall Lemma~\ref{l_dig_a_trap}.
Observe that, if $\eta\in\WW_1^{\xi_0}$ and $\ell\in\Phi$, then
$\eta+\ell\in\WW_2^{\xi_0}$, and also if $\eta_1+\ell_1=\eta_2+\ell_2$
for $\eta_{1,2}\in\WW_1^{\xi_0}$ and $\ell_{1,2}\in\Phi$,
then $\eta_1=\eta_2$ and $\ell_1=\ell_2$.
With these observations, we write for any~$\xi_0$
such that $\xi_0(0)>0$ and $\xi_0(x)>0$
\begin{align*}
 \bP\big[\TT (x,M,n), \I^u_{\CC_M(n)}=\xi_0\big]
 & = \sum_{\eta\in\WW_2^{\xi_0}}\bP[\LL^u_A = \eta]\\
&\geq \sum_{\eta'\in\WW_1^{\xi_0}} \sum_{\ell\in \Phi}
 \IP[\LL^u_A = \eta'+\ell]\\
&\geq \sum_{\eta'\in\WW_1^{\xi_0}}\IP[\LL^u_A = \eta']
 \sum_{\ell\in \Phi} (2d)^{-|\ell|}\\
&\geq n^{-\gamma/M}
\IP\big[Q(\xe,M,n) \text{ is vacant}, \I^u_{\CC_M(n)}=\xi_0\big]\\
&\geq n^{-\gamma/M} 
\exp\Big(-\frac{g_M\ln n}{\ln\ln n}\Big)
\IP\big[\I^u_{\CC_M(n)}=\xi_0\big],
\end{align*}
where the last inequality follows from Lemma~\ref{l_find_quiver}
and the second-to-last one from  Lemma~\ref{l_dig_a_trap}.
This implies that
\begin{equation}
\label{eq_find_trap}
 \bP_0\big[\TT (x,M,n)\mid \I^u_{\CC_M(n)}, x\in \I^u_{\CC_M(n)}\big]
   \geq n^{-\gamma/M},
\end{equation}
for any $x\in\partial^+\CC_M(n)$.

\section{Proofs of the main theorems}
\label{s_proofs_main}

\subsection{The biased random walk on the interlacement set in 
three dimensions has sub-polynomial speed}
\label{s_proof_subpol}

In this section we prove Theorem~\ref{theorem_3d}.

Consider a sequence of cones $\CC_M(jn^{1/3})$, $j=1,\ldots,n^{2/3}$,
and let $\tau_j=T_{\partial \CC_M(jn^{1/3})}$,
see Figure~\ref{f_trapping}. 
Recall the definition of the ``good'' environment from 
Lemma~\ref{l_find_quiver}, and define 
a decreasing sequence of events
\begin{equation}
\label{def_hat_GG}
\hG_k = \bigcap_{j=1}^{k}
\bigcap_{x\in \partial^+ \CC_M(j n^{1/3})} \GG(x,M,jn^{1/3})
\end{equation}
for $k\leq n^{2/3}$; let also $\hG:=\hG_{n^{2/3}}$.
Observe that Lemma~\ref{l_find_quiver} implies that
\begin{equation}
\label{prob_GG}
 \bP[\hG] \geq 1- \sum_{k=1}^{n^{2/3}}M^2
 (jn^{1/3})^2\times (jn^{1/3})^{-12}
   \geq 1-C M^2 n^{-3} .
\end{equation}

Next, let us define
\[
 \zeta = \begin{cases}
          \infty, & \text{ on }\hG,\\
          j, & \text{ on }\hG_j\setminus \hG_{j+1},
            \text{ if } j<n^{2/3},\\
             0,  & \text{ on }\hG_1^\complement,
         \end{cases}
\]
and set formally $\tau_0=0, \tau_\infty=\infty$.
We introduce another process~$\hX$ in the following
way: for $k\geq 0$
\[
 \hX_k = \begin{cases}
          X_k, & \text{ for }k\leq \tau_\zeta,\\
           X_{\tau_\zeta}, & \text{ for }k > \tau_\zeta,
         \end{cases}
\]
i.e., it is equal to the old process~$X$ until
the process stays in the ``good'' cone  $\CC_M(\zeta n^{1/3})$,
and then is stopped.
We define also $\hat{\tau}_j={\widehat T}_{\partial \CC_M(jn^{1/3})}$,
where ${\widehat T}$'s are the hitting times for~$\hX$.
Then, let us define a sequence of events
\begin{equation}
\label{def_hH}
 \hH_j = \begin{cases}
           \Omega, &\text{on }\hat{\tau}_j=\infty 
           \text{ or when there exists }
               k\leq j\\
               & \qquad \text{ such that } 
           \hX_{\hat{\tau}_j}\in\partial^-\CC_M(kn^{1/3}),\\
             \{\hat{\tau}_{j+1}-\hat{\tau}_{j} 
              > n^{\frac{1}{3}M\ln \beta}\},&\text{otherwise}.
         \end{cases}
\end{equation}
Let $\widehat{\mathcal{F}}_{\hat{\tau}_j}$ be the sigma-algebra
generated by $\hX_0,\ldots,\hX_{\hat{\tau}_j}$.

We start by showing that when exiting a cone $\CC_M(jn^{1/3})$,
conditionally on any type of past 
information which was likely to occur, we have a decent chance of 
spending a lot of time in $\CC_M((j+1)n^{1/3})$:
\begin{lemma}
\label{lem_deeptrap}
Fix any $M<\infty $ and $\beta >1$.  
 We have for all $j=1,\ldots,n^{2/3}$
\[
\IP\big[\hH_j \mid \widehat{\mathcal{F}}_{\hat{\tau}_j}\big]
    \geq n^{-2\gamma /M},
\]
 for all~$n$ large enough.
\end{lemma}

\begin{figure}
 \centering \includegraphics{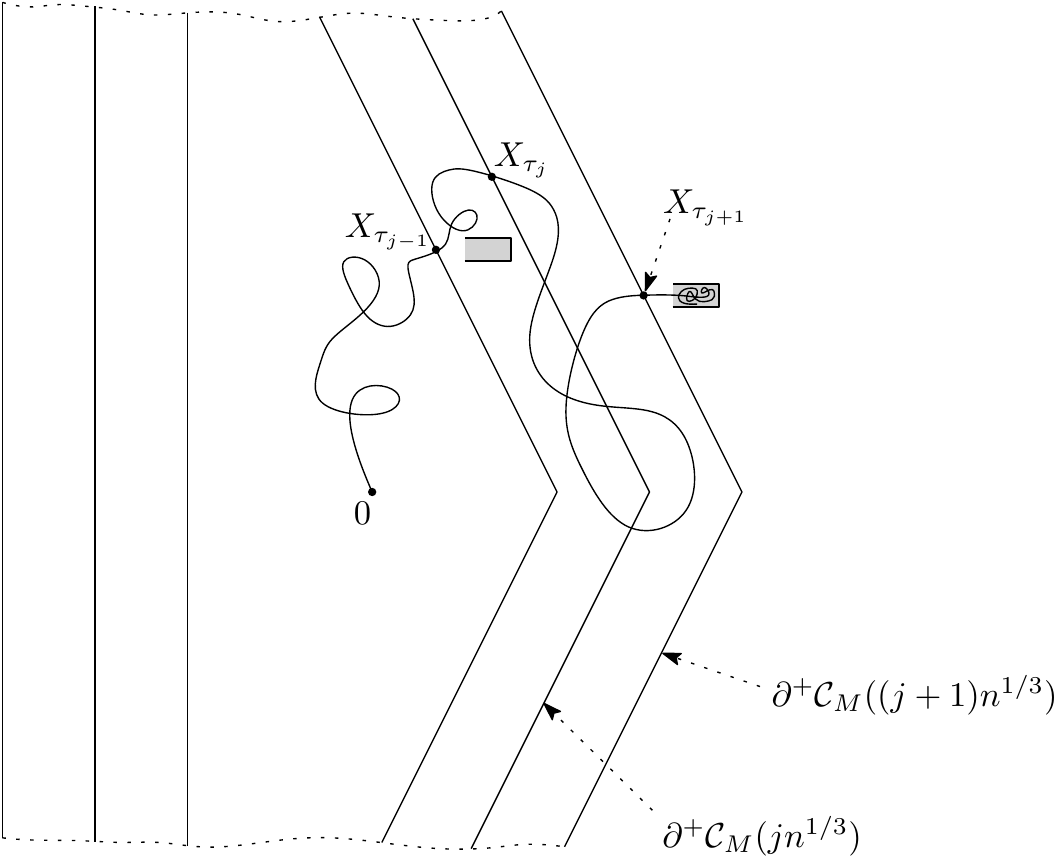} 
\caption{On the proof of Theorem~\ref{theorem_3d}.
There is a trap in front of~$X_{\tau_{j-1}}$, but the walk
manages to escape it; there is no trap in front of~$X_{\tau_j}$;
there is again a trap in front of~$X_{\tau_{j+1}}$
and the walk finally gets caught there.}
\label{f_trapping}
\end{figure}

\begin{proof}
Assume that $x \in \partial^+\CC_M(jn^{1/3})$.
In the following, we abbreviate
$\xe:=\xe(x,M,jn^{1/3})$ and $\xt:=\xt(x,M,jn^{1/3})$.
On the event $\TT(x,M,jn^{1/3})$, we know that 
\begin{enumerate}
\item $x+ke_1\in\mathcal{I}_u, \text{ for } 
      0\leq k\leq \lf \frac{3}{M}\ln n \rf$;
\item $\xe $ is $\mathcal{I}_u$-connected to 
$\xt$ inside $Q(\xe, M, jn^{1/3})$;
\item the connected component of $\xt$ in 
$\mathcal{I}_u \setminus \xe $ is finite.
\end{enumerate}

In particular, the first property above implies
\begin{equation}\label{deeptrap1}
\Po_x\big[X_{1}=x+e_1,\ldots, X_{\lf\frac{3}{M} \ln n\rf}
 =x+\lf{\textstyle\frac{3}{M}} \ln n\rf e_1\big] 
\geq \Bigl(\frac \beta{\beta +2d-1}\Bigr)^{\lf\frac 3M \ln n\rf}.
\end{equation}



Also on the event $\TT(x,M,jn^{1/3})$, 
by the second and third property 
above, we know that the connected component of $\xt$ 
in $\mathcal{I}_u\setminus \xe $ is finite and 
is adjacent to~$\xe $. We denote by~$G$ the finite 
network formed by this connected component and~$\xe$.
Let us also denote by~$\mathrm{P}^G$ the probability
for the walk restricted on~$G$.

 On $\TT(x,M,jn^{1/3})$ it is only possible to exit 
$Q(\xe, M, jn^{1/3})$ through~$\xe$;
also, note that we defined traps in such a way that from~$\xe$
the particle can jump only to the left or to the right, 
the jumps in the transversal directions cannot happen since 
the corresponding sites are not in the interlacement set.
Hence, the jump from~$\xe$ to the right happens
with probability $\frac{\beta}{\beta+1}$, and we can write
\begin{align}\label{deeptrap2}
 \Po_{\xe }\big[T_{\xt}<T_{Q(\xe, M, jn^{1/3})}\big]
&=  
\frac{\beta}{\beta+1}
\mathrm{P}^G_{\xe }[T_{\xt}<T_{\xe}].
\end{align}

Using the notation  $\pi^G(x)=\sum_{y\in G} c(x,y)$
and~$C^G$ for the effective conductance in~$G$, 
we can use some standard facts about electrical network theory
(see e.g.~(2.4) in~\cite{LP}) to obtain that
\[
\mathrm{P}^G_{\xe }[T_{\xt}<T_{\xe}]
=\frac{ C^G(\xe  
\leftrightarrow \xt)}{\pi^G(\xe )},
\]
and it is easy to see that 
$\pi^G(\xe )\leq \gamma  \beta^{x\cdot e_1 
+\lf\frac{3}{M}\ln n^{1/3}\rf}
=\gamma\beta^{x\cdot e_1 +\lf\frac{1}{M}\ln n\rf}$
(recall that~$x$ is a site on~$\partial^+\CC_M(jn^{1/3})$).  
Furthermore, on $\TT(x,M,jn^{1/3})$, there is a simple path linking~$\xe$
to~$\xt$ of length at most $\gamma(M\ln n)^3$ edges 
all with conductances at least $\beta^{x\cdot e_1 +\lf\frac{1}{M}\ln n\rf}$. 
Rayleigh's monotonicity principle then implies that 
\[
 C^G(\xe  \leftrightarrow \xt) \geq 
 \frac{\beta^{x\cdot e_1 +\lf\frac{1}{M}\ln n\rf}}{\gamma(M\ln n)^3},
 \]
 so that 
 \[
\mathrm{P}^G_{\xe }[T_{\xt}<T_{\xe }]
\geq \gamma(M\ln n)^{-3}.
 \]
  The previous inequality along with~(\ref{deeptrap1}) 
  and~(\ref{deeptrap2}) imply that on $\TT(x,M,jn^{1/3})$
 \begin{equation}
\label{deeptrap3}
 \Po_x\big[T_{\xt}<T_{\partial \CC_M((j+1)n^{1/3})} \big]
 \geq n^{-\gamma /M}.
 \end{equation}
 
 Moreover, on $\TT(x,M,jn^{1/3})$, we know that starting 
  from~$\xt$ we need to reach~$\xe $ before we can 
 exit $\CC_M((j+1)n^{1/3}))$. 
 Furthermore, we see by reversibility that
 \[
 \Po_{\xt}[T_{\xe }<T_{\xt}] = \frac{\pi^G(\xe )}{\pi^G(\xt)}  
 \Po_{\xe }[T_{\xt}<T_{\xe}]\leq \gamma  \beta^{-\frac{M}{3}\ln n},
 \]
 so the number of returns to $\xt$ before exiting 
 $\CC_M((j+1)n^{1/3}))$ is a geometric random variable 
 of parameter at most $\gamma \beta^{-\frac{M}{3}\ln n}$. This means 
 that the time to exit $\CC_M((j+1)n^{1/3})$ is larger than a
 geometric random variable of parameter
 at most $\gamma \beta^{-\frac{M}{3}\ln n}$,
 so for some uniformly positive~$\eps$ we have
 \[
 \Po_{\xt}[T_{\partial \CC_M((j+1)n^{1/3})} \geq \beta^{M\ln n}] 
 >\eps>0.
 \]
 
 This result along with~(\ref{deeptrap3}) and the Markov
 property implies that on any environment 
belonging to~$\TT(x,M,jn^{1/3})$
\begin{equation}
\label{quenched_delay}
 \Po_{x}[T_{\partial \CC_M((j+1)n^{1/3})} \geq \beta^{M\ln n}]
   \geq  n^{-\gamma /M}.
\end{equation}
  
Having dealt with the quenched probabilities, 
we move on.
For any finite sequence ${\tilde x}=(x_0,x_1,\ldots, x_m)$
of sites in~$\Z^3$ define
\[
  \Gamma_\omega({\tilde x}) 
   = \Po_{x_0}\big[\hX_1=x_1,\ldots,\hX_m=x_m\big].
\]
Let $x_1,\ldots,x$ be a sequence of sites in $\CC_M(jn^{1/3})$,
and assume that $x \in \partial^+\CC_M(jn^{1/3})$.
Write 
\begin{align}
 \lefteqn{\IP\big[\hH_j \mid \hX_1=x_1,\ldots,\hX_{\tau_j}=x\big]}
\nonumber\\
 &=
\frac{\bE_0\big(\Po_0[\hH_j \mid \hX_1=x_1,\ldots,\hX_{\tau_j}=x] 
\Gamma_\omega(0,x_1,\ldots,x)\big)}
{\bE_0 \Gamma_\omega(0,x_1,\ldots,x)}\nonumber\\
&=
\frac{\bE_0\big(\Po_x[\hH_j] 
\Gamma_\omega(0,x_1,\ldots,x)\big)}
{\bE_0 \Gamma_\omega(0,x_1,\ldots,x)}
\nonumber\\
&\geq \frac{\bE_0\big(\Po_x[\hH_j]\1{\omega\in\TT(x,M,jn^{1/3})}
\Gamma_\omega(0,x_1,\ldots,x)\big)}
{\bE_0 \Gamma_\omega(0,x_1,\ldots,x)}.
\label{annealed_conditional}
\end{align}
Note the following general fact: if $\xi,\eta \geq 0$ are random
variables, $\eta$ is measurable with respect to a 
sigma-algebra~$\mathcal{A}$, and $\bE(\xi\mid \mathcal{A})\geq \gamma_1$,
then $\bE (\xi\eta) \geq \gamma_1 \bE\eta$.
Let $\mathcal{A}^u_A$ be the sigma-algebra generated by the 
random interlacements of level~$u$ on the set~$A\subset\Z^3$.
Inequalities~\eqref{eq_find_trap} and~\eqref{quenched_delay} imply that
\[
 \bE\big(\Po_x[\hH_j]\1{\omega\in\TT(x,M,jn^{1/3})}
\mid \mathcal{A}^u_{\CC_M(jn^{1/3})}\big) \geq n^{-2\gamma/M},
\]
and, since $\Gamma_\omega(0,x_1,\ldots,x)$ is clearly
$\mathcal{A}^u_{\CC_M(jn^{1/3})}$-measurable, we finish
the proof of Lemma~\ref{lem_deeptrap} using~\eqref{annealed_conditional}
and the above general fact.
\end{proof}

\begin{lemma}
\label{lem_end}
There exists $\gamma_1>0$ such that 
\[
\IP\big[T_{\partial \CC_M(n)}\geq n^{\frac{1}{3}M\ln \beta}\big]
\geq 1-n^{-\gamma_1}.
\]
\end{lemma}

\begin{proof}
First, observe that Lemma~\ref{lem_deeptrap} implies
that 
\begin{equation}
\label{n2/3}
 \IP\Big[\bigcap_{j=1}^{n^2/3} \hH_j^\complement\Big]
\leq \exp\big(-\gamma' n^{\frac{2}{3}-\frac{\gamma}{M}}\big)
\end{equation}
(indeed, we have $n^{2/3}$ tries with success probability 
at least~$n^{-\gamma/M}$, independently of the past).
 
Now, recalling the notation of $\Phi_n$ at~\ref{def_Phi_n}, write
\begin{align*}
 \IP\big[T_{\partial \CC_M(n)} < n^{\frac{1}{3}M\ln \beta}\big]
&\leq \IP\big[T_{\partial \CC_M(n)} < n^{\frac{1}{3}M\ln \beta},\Phi_n\big]
   + \IP[\Phi_n^\complement]\\
&\leq \bE\big(\Po\big[T_{\partial \CC_M(n)} 
< n^{\frac{1}{3}M\ln \beta},\Phi_n\big]\1{\hG}\big)
 + \bP[\hG^\complement]
   + \IP[\Phi_n^\complement]\\
&= \bE\Big(\Po\Big[\bigcap_{j=1}^{n^{2/3}} \hH_j^\complement\Big]
   \1{\hG}\Big)
+ \bP[\hG^\complement]
   + \IP[\Phi_n^\complement]\\
&\leq \IP\Big[\bigcap_{j=1}^{n^{2/3}} \hH_j^\complement\Big]
+ \bP[\hG^\complement]
   + \IP[\Phi_n^\complement],
\end{align*}
and we use~\eqref{est_Phi_n}, \eqref{prob_GG},
and~\eqref{n2/3} to conclude the proof of Lemma~\ref{lem_end}.
%
\end{proof}

We now finish the proof of Theorem~\ref{theorem_3d}. 
Indeed, Lemma~\ref{lem_end} together with Borel-Cantelli's lemma
imply that 
\begin{equation}
\label{denouement}
 \IP\big[T_{\partial \CC_M(2^k)}
   \geq 2^{\frac{k}{3}M\ln \beta} \text{ for almost all }k\big] = 1.
\end{equation}
So, for all large 
enough~$t\in \big[2^{\frac{k}{3}M\ln \beta}, 
2^{\frac{k+1}{3}M\ln \beta}\big)$ we have $X_t\in \CC_M(2^{k+1})$.
Since $y\in \CC_M(n)$ implies $\|y\| \leq M n$,
we obtain
\[
 \limsup_{t\to\infty}\frac{\ln\|X_t\|}{\ln t}
 \leq \limsup_{k\to\infty}\frac{(k+1)\ln 2 + \ln M}
{\frac{1}{3}kM\ln\beta \ln 2 } = \frac{3}{M\ln \beta} 
\qquad \text{$\IP$-a.s.,}
\]
which proves Theorem~\ref{theorem_3d} since~$M$ is arbitrary.
\qed

\subsection{Dimension $d\geq 4$}
\label{s_d4}

In this section we prove Theorem~\ref{t_big_d_0speed}.

\begin{figure}
 \centering \includegraphics{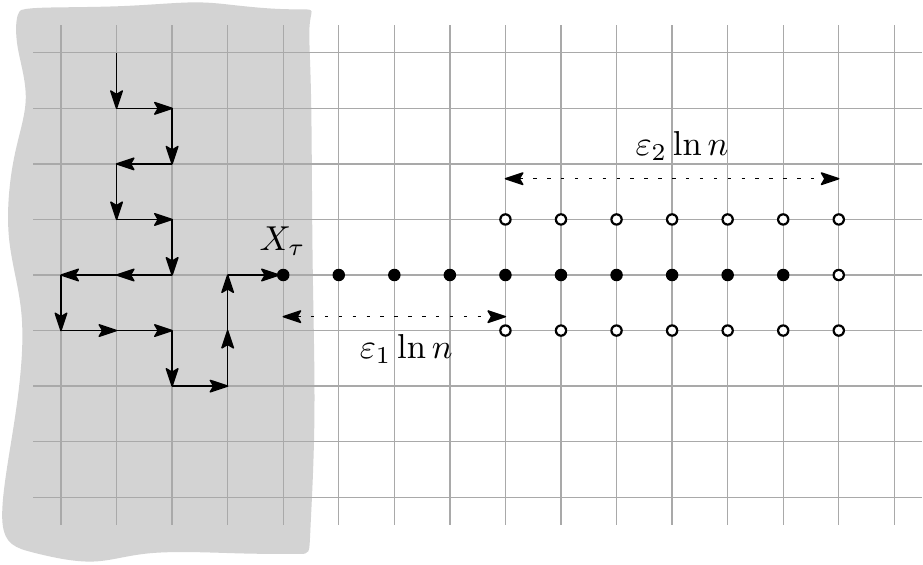} 
\caption{Traps in higher dimensions ($\eps_{1,2}$ are supposed to
be small enough). 
The shaded area corresponds
to the already explored part of the environment; $\bullet$'s belong 
to the interlacement set, and $\circ$'s are vacant
(the states of other sites can be arbitrary). }
\label{f_d4trap}
\end{figure}

The proof of this result is \emph{very} similar to the 
proof of Theorem~\ref{theorem_3d}, so we only indicate where
changes have to be made. The point is that, in dimensions~$d\geq 4$,
it costs too much to have a trap as on Figure~\ref{f_trap}.
Instead, we use a simpler kind of traps, see Figure~\ref{f_d4trap}.
 When the particle faces a yet unexplored region, we just 
ask that there is a straight segment of sites belonging to
the interlacement set of length~$\ln n$ times a small constant,
and the rightmost part of this segment is surrounded by vacant 
sites, as shown on the picture. It can be shown that the capacity
of the ``quiver'' of the vacant sites is approximately $\gamma \eps_2\ln n$,
so the cost of having this quiver in the environment 
is roughly $n^{-\gamma'\eps_2}$, that is, power in~$n$, but with
a small power. 
Then, it can be easily seen that the 
cost of having the straight segment of occupied sites is similar,
roughly $n^{-\gamma''(\eps_1+\eps_2)}$. 
Also, the decoupling argument works even better because of the~$s^{d-2}$ in Proposition~\ref{p_main3}.
So, it does not cost more than $n^{-\gamma''(\eps_1+\eps_2)}$ to have a trap like this each time
when the particle faces the unexplored region. Now, regardless of our choice of $\epsilon_{1,2}$, 
it is clear that if~$\beta$ is very large, then the walk will
spend a lot of time (say, at least~$n^2$) in the trap with 
probability at least~$n^{-\gamma_1}$, where~$\gamma_1$ 
can be made arbitrarily small by decreasing~$\eps_{1,2}$.
This shows the result.
\qed

\section*{Acknowledgements}
Part of this work was done during the visit of the first 
author to IMECC--UNICAMP, supported by FAPESP (grant 2012/07166--9).
The first author is thankful to NSERC and FRQNT for financial support.
The second author thanks
CNPq (grant 300886/2008--0) 
for financial support. 
Both authors thank the referee for careful
reading of the manuscript and valuable comments and suggestions.


\begin{thebibliography}{19}

\bibitem{AP}
Alves, C., and Popov, S. (2015)
 Conditional decoupling of random interlacements.
arXiv:1508.03405
     
\bibitem{BB} Berger, N. and Biskup, M. (2007). 
Quenched invariance principle for simple random walks on percolation clusters.
\textit{Probab.\ Theory Relat.\ Fields}. \textbf{130} (1--2), 83--120.


\bibitem{BGP} Berger, N., Gantert, N., and Peres, Y. (2003)
 The speed of biased random walk on percolation clusters. 
\textit{Probab.\ Theory Relat.\ Fields} \textbf{126} (2), 221--242.


\bibitem{CP} 
Cern\'y, J., and Popov, S. (2012)
 On the internal distance in the interlacement set. 
\textit{Electron.\ J.\ Probab.} \textbf{17}, paper No.~29, 1--25.

\bibitem{CT12} Cern\'y, J., and Teixeira, A. (2012)
  From random walk trajectories to random interlacements. 
Ensaios Matem\'aticos [Mathematical Surveys] \textbf{23}. 
Sociedade Brasileira de Matem\'atica, Rio de Janeiro.





\bibitem{DRS14} Drewitz, A., R\'ath, B, and Sapozhnikov, A. (2014)
\textit{An introduction to random interlacements}.
Springer.

\bibitem{Fribergh} Fribergh, A. (2010) 
The speed of a biased random walk on a 
percolation cluster at high density.
 \textit{Ann.\ Probab.} \textbf{38} (5), 1717--1782.

\bibitem{FH} Fribergh, A., and Hammond, A. (2014)
 Phase transition for the speed of the biased random walk 
on a supercritical percolation cluster. 
\textit{Commun.\ Pure Appl.\ Math.} \textbf{67} (2), 173--245.



 
\bibitem{Law91}
Lawler, G. (1991) 
\textit{Intersections of random walks}.
 Probability and its Applications, Birkh\"auser Boston.

\bibitem{LL10} 
Lawler, G, and Limic, V. (2010)
\textit{Random walk: a modern introduction.} 
Cambridge Studies in Advanced Mathematics, \textbf{123}. 
Cambridge University Press, Cambridge.

\bibitem{LP} Lyons, R., and Peres, Y. (2016)
\textit{Probability on trees and networks}. 
Cambridge University Press.

\bibitem{MP} Mathieu, P. and Piatnitski, A. (2007). 
Quenched invariance principles for random walks on percolation clusters.
\textit{Proceedings Royal Soc.\ A} \textbf{463}, 2287--2307.


 
\bibitem{SLT} Popov, S., and Teixeira, A. (2015)
Soft local times and decoupling of random interlacements. 
\textit{J.\ European Math. Soc.} \textbf{17} (10), 2545--2593.

\bibitem{PRS15} Procaccia, E., Rosenthal, R., Sapozhnikov, A. (2016)
Quenched invariance principle for simple random walk on clusters 
in correlated percolation models.
\textit{Probab.\ Theory Relat\ Fields}.
\textbf{166} (3), 619--657.

\bibitem{Spitzer} Spitzer, F. (1976)
\textit{Principles of random walk.}
Springer, New York.

\bibitem{Stern} Stern, F. (1975)
Conditional expectation of the duration in the classical ruin problem.
\textit{Math.\ Mag.} \textbf{48} (4), 200--203.

\bibitem{SS} V. Sidoravicius and A.-S. Sznitman. (2004) Quenched invariance principles for walks on clusters of percolation or among random conductances.
\textit{Probab.\ Theory Relat.\ Fields} \textbf{129} (2), 219--244.

\bibitem{Sznitman} Sznitman, A.-S. (2003)
 On the anisotropic random walk on the percolation cluster. 
\textit{Commun. Math. Phys.}. \textbf{240 (1-2)}, 123--148.

\bibitem{sznitmannew} Sznitman, A.-S. (2006)
Random motions in random media. 
\textit{Mathematical statistical mechanics.} 219--242. 
Elsevier B. V., Amsterdam.

\bibitem{SZ2} Sznitman, A.-S. (2004)
\textit{Topics in random walks in random environment}. 
School and Conference on Probability Theory, 
ICTP Lecture Notes Series, Trieste, \textbf{17}, 203--266.

\bibitem{Szn10} Sznitman, A.-S. (2010)
Vacant set of random interlacements and percolation.
\textit{Ann.\ Math.\ (2)}, \textbf{171} (3), 2039--2087.


\end{thebibliography}
\end{document}